\documentclass[a4paper,11pt]{amsart}
\usepackage{enumitem}   
\usepackage{cite}
\usepackage{tikz-cd}
\usetikzlibrary{er,positioning}
\usepackage{amsmath,amssymb,amsthm,graphicx,yfonts,mathrsfs,color,mathtools}
\usepackage{amscd}  
\usepackage{stmaryrd,colonequals}
\usepackage[all]{xy} 
\usepackage{multirow} 
\usepackage{array}
\usepackage[utf8]{inputenc}
\usepackage{braket}
\usepackage{rotating}
\usepackage{setspace,comment}
\usepackage{multirow}
\usepackage{rotating}
\usepackage{pifont}
\usepackage{listings}
\PassOptionsToPackage{hyphens}{url}
\usepackage[colorlinks, backref, breaklinks]{hyperref}
\hypersetup{%
  colorlinks = true,
  linkcolor  = black
}

\usepackage{xcolor}

\definecolor{codegreen}{rgb}{0,0.6,0}
\definecolor{codegray}{rgb}{0.5,0.5,0.5}
\definecolor{codepurple}{rgb}{0.58,0,0.82}
\definecolor{backcolour}{rgb}{0.95,0.95,0.92}

\lstdefinestyle{mystyle}{
  backgroundcolor=\color{white},   commentstyle=\color{codegreen},
  keywordstyle=\color{magenta},
  numberstyle=\tiny\color{codegray},
  stringstyle=\color{codepurple},
  basicstyle=\ttfamily\footnotesize,
  breakatwhitespace=false,         
  breaklines=true,                 
  captionpos=b,                    
  keepspaces=true,                 
  numbers=left,                    
  numbersep=5pt,                  
  showspaces=false,                
  showstringspaces=false,
  showtabs=false,                  
  tabsize=2
}
\usepackage[colorinlistoftodos]{todonotes}
\usepackage{thmtools}
\usepackage{thm-restate}
\theoremstyle{plain}
\newtheorem{thm}{Theorem}[section]
\newtheorem{lemma}[thm]{Lemma}
\newtheorem{proposition}[thm]{Proposition}
\newtheorem{prop}[thm]{Proposition}
\newtheorem{cor}[thm]{Corollary}

\newtheorem{claim*}{Claim}

\theoremstyle{definition}

\newtheorem{rmk}[thm]{Remark}

\newtheorem{assumption}{Assumption}

\makeatletter
\def\th@plain{%
  \thm@notefont{}
  \itshape 
}
\def\th@definition{%
  \thm@notefont{}
  \normalfont 
}
\makeatother

\newcommand{\ordnop}{\mathop{\mathrm{ord}}\nolimits}

\newcommand{\ord}{\operatorname{ord}}

\DeclareMathOperator{\Frac}{Frac}
\DeclareMathOperator{\known}{known}
\DeclareMathOperator{\tors}{tors}

\newcommand{\Q}{\mathbb{Q}}

\newcommand{\Z}{\mathbb{Z}}
\newcommand{\F}{\mathbb{F}}

\newcommand{\Zp}{\mathbb{Z}_{p}}

\renewcommand{\div}{\operatorname{div}}
\newcommand{\Log}{\mathop{\mathrm{Log}}\nolimits}

\newcommand{\rank}{\operatorname{rank}}

\newcommand{\Jac}{\operatorname{Jac}}

\renewcommand{\div}{\operatorname{div}}

\DeclareMathOperator{\dR}{\operatorname{dR}}
\DeclareMathOperator{\red}{red}

\begin{document}

\title{Rational points on rank 2 genus 2 bielliptic curves in the LMFDB}

\author{\sc Francesca Bianchi}
\address{Francesca Bianchi \\
Bernoulli Institute for Mathematics, Computer Science and Artificial Intelligence, University of Groningen, Groningen, The Netherlands}
\email{francesca.bianchi@rug.nl}

\author{\sc Oana Padurariu}
\address{Oana Padurariu \\
Dept. of Mathematics \& Statistics\\  
Boston University\\
USA}
\urladdr{https://sites.google.com/view/oana-padurariu/home}
\email{oana@bu.edu}

\date{\today}
\maketitle
\begin{abstract}
Building on work of Balakrishnan, Dogra, and of the first author, we provide some improvements to the explicit quadratic Chabauty method to compute rational points on genus $2$ bielliptic curves over $\Q$, whose Jacobians have Mordell--Weil rank equal to $2$. 
We complement this with a precision analysis to guarantee correct outputs. Together with the Mordell--Weil sieve, this bielliptic quadratic Chabauty method is then the main tool that we use to compute the rational points on the $411$ locally solvable curves from the LMFDB which satisfy the aforementioned conditions.
\end{abstract}

\setcounter{tocdepth}{1}
\tableofcontents

\section{Introduction}
Let $X$ be a smooth, projective, geometrically integral curve over the field of rational numbers; let $g$ be its genus, and let $r$ be the Mordell--Weil rank of its Jacobian $J$ over $\Q$.
In this article we compute the set of rational points on all such $X$ in the LMFDB \cite{lmfdb} - and, in particular, in the database of genus 2 curves computed by Booker, Sijsling, Sutherland, Voight and Yasaki \cite{genus2curvedatabase} - which satisfy the following conditions:
\begin{enumerate}[label=(\roman*)]
\item\label{ass:1} $g=2$;
\item\label{ass:2} $r= 2$;
\item\label{ass:3} $X$ is bielliptic over $\Q$.
\end{enumerate}
For some of these curves, there exists a place $v$ of $\Q$ for which $X(\Q_v) = \emptyset$, so the set of rational points  $X(\Q)$ is trivially empty. After discarding these, we are left with $411$ curves that satisfy \ref{ass:1}--\ref{ass:3} and are everywhere locally solvable. Our main contribution, which made the computation of the rational points on these curves  possible, is an improvement of the explicit quadratic Chabauty approach specific to curves satisfying \ref{ass:1}--\ref{ass:3} and a \texttt{SageMath} \cite{sage} implementation of the resulting method (available at \cite{OurCode}). We now explain how the assumptions \ref{ass:1}--\ref{ass:3} place our problem into the more general context of computing rational points on curves.

First of all, since we are assuming [\ref{ass:1}] that the genus of $X$ is greater than $1$, the set of rational points $X(\Q)$ is finite by Faltings' theorem \cite{Faltings83, faltingserratum}.  Without any further assumption on $X$, there is no practical algorithm that is guaranteed to provably output all the finitely many points in $X(\Q)$. 

If $r=0$, determining $X(\Q)$ is easy; if $r<g$ (and $g$ is small, as in our case), a combination of the Chabauty--Coleman method \cite{Chabauty41, Coleman85} and the Mordell--Weil sieve \cite{Scharaschkin99, BS10} is likely to be successful.
Let $p$ be a prime of good reduction for $X$. The core idea of the Chabauty--Coleman method is that the $p$-adic closure of $J(\Q)$ inside the $p$-adic manifold $J(\Q_p)$ has codimension at least $\max\{g-r,0\}$, and hence positive codimension if $r<g$. Pulling back to $X$ via $X\to J$, $x\mapsto [\deg D \cdot x - D]$ for a $\Q$-rational divisor $D$ of positive degree, this allows one to write down a locally analytic function $\tilde{\rho}_0\colon X(\Q_p)\to \Q_p$ which vanishes on $X(\Q)$. The zero set $A_0$ of $\tilde{\rho}_0$ is finite and contains $X(\Q)$; the Mordell--Weil sieve can then often be used to extract  from $A_0$ the set $X(\Q)$. An implementation of this method for $g=2$ is available in \texttt{Magma} \cite{magma}.

When $r\geq g$, the Chabauty--Coleman method is, in general, not applicable and computing $X(\Q)$ is often a harder problem. Our geometric assumption \ref{ass:3} comes in to simplify the task. Recall that a curve is bielliptic if it has a degree $2$ map to an elliptic curve. We will further require that the map is defined over the base field of the curve. Then a bielliptic genus $2$ curve over $\Q$ admits a model of the form
\begin{equation*}
    y^2 = a_6x^6 +a_4x^4 +a_2 x^2+ a_0,\qquad a_i\in \Z,
\end{equation*}
and its Jacobian is isogenous to $E_1\times E_2$, where $E_1$ and $E_2$ are elliptic curves given by the following Weierstrass equations:
\begin{align*}
E_1\colon y^2 = x^3 + a_4x^2 + a_2a_6x + a_0a_6^2\\
E_2 \colon y^2 = x^3 + a_2x^2 + a_4a_0x + a_6 a_0^2
\end{align*}
(see \cite{FK91, Kuhn88}). If one of $E_1$ and $E_2$ has Mordell--Weil rank $0$ over $\Q$, we can easily compute the set $X(\Q)$; the more interesting case is when the rank of the Jacobian of $X$ is at least $2$ because each of $E_1$ and $E_2$ has positive rank.

For example, $X_w/\Q\colon y^2 = x^6+x^2+1$ is a bielliptic genus $2$ curve, whose elliptic quotients each have rank $1$  (so the rank of the Jacobian of $X_w$ is $2$). It turns out that the determination of $X_w(\Q)$ is equivalent to solving Problem 17 of book VI of Diophantus' Arithmetica. Wetherell \cite{wetherell} observed that, in this case, one could exploit the isogeny $E_1\times E_2 \sim \Jac(X_w)$ to reduce the problem of computing $X_w(\Q)$ to that of computing the rational points on two genus $3$ curves, for which the method of Chabauty--Coleman is applicable. By carrying this out explicitly, he solved Diophantus' problem, many centuries after it had been formulated. 

Flynn and Wetherell \cite{Flynn_wetherell} recast Wetherell's solution to Diophantus' problem as a special case of a strategy that can be applied to compute the rational points on arbitrary bielliptic genus $2$ curves whose corresponding elliptic curve quotients each have rank equal to $1$. Furthermore, they replaced the Chabauty--Coleman computations of Wetherell with computations on elliptic curves over number fields. These have hope of being successful only if the rank of such elliptic curves is strictly less than the number field degree. They called the resulting elliptic curve computations ``elliptic curve Chabauty''; an extension of the genus $2$ bielliptic method of Flynn--Wetherell to curves covering elliptic curves (possibly over some extension of $\Q$) is due to Bruin \cite{Bruin}. 
It would be interesting to investigate for how many of the curves in our database the computation of rational points is algorithmically possible using elliptic curve Chabauty. We have not attempted this, but mention in this respect that Hast \cite{Hast_2_cov} has recently implemented a method to compute rational points on genus $2$ curves with a rational Weierstrass point, which combines descent and elliptic curve Chabauty. The resulting algorithm was run on a large database of curves, 21 of which also belong to our database. 
 For 15 of these 21 the computation was not successful (see \cite{HastRawData}), due to current algorithmic limitations of \texttt{Magma} \cite{magma}, for instance in the computation of Mordell--Weil ranks of elliptic curves over non-trivial extensions of $\Q$. This does not allow one to conclude whether in some of these cases there could be an actual theoretical obstruction to the method.

In a different direction, Kim's program \cite{KimP1, Kimunipotent} aims to construct, for an arbitrary $X/\Q$, locally analytic functions $\tilde{\rho }\colon X(\Q_p) \to \Q_p$ vanishing on the set $X(\Q)$ by replacing the Jacobian in the method of Chabauty--Coleman with suitable Selmer varieties. Balakrishnan and Dogra \cite{BD18} made one level of Kim's program explicit for curves satisfying $r< g+\rho - 1$, where $\rho$ is the rank of the N\'eron--Severi group of $J$ over $\Q$. The resulting method is known as ``quadratic Chabauty'', and the locally analytic function $\tilde{\rho}$ in this case is defined using quadratic forms constructed from $p$-adic heights.
\begin{rmk}\label{rmk:QC_integral}
Here by ``quadratic Chabauty'' we always mean, unless otherwise specified, quadratic Chabauty for rational points. A simpler variant of quadratic Chabauty can be used to determine the integral points of elliptic and hyperelliptic curves whose genus is equal to the rank of the Jacobian, and was developed prior to the work of Balakrishan--Dogra.  See \cite{KimMasseyProducts, AppendixToMasseyProduct, ColemanGrosspAdicSigma} for elliptic curves and \cite{QC0,BBM17} for hyperelliptic curves. 
\end{rmk}
In particular, the quadratic Chabauty method is applicable to curves satisfying \ref{ass:1}--\ref{ass:3}, provided that $E_1$ and $E_2$ each have rank $1$; in fact, these are perhaps the simplest class of curves for which Chabauty's method is not applicable, but quadratic Chabauty is. Not surprisingly, the first explicit examples of quadratic Chabauty in the literature are genus $2$ bielliptic curves\cite[\S\S 8.3,8.4]{BD18}. By now, Balakrishnan--Dogra's quadratic Chabauty has also been successfully applied to many modular curves of arithmetic interest, using the techniques and algorithms of  \cite{BDMTV19, QCmod}. 

The bielliptic genus $2$ case still remains interesting, since it can be understood independently of the $p$-adic Hodge theory techniques that are normally involved in quadratic Chabauty. In this case, the locally analytic function $\tilde{\rho}$ is defined using abelian integrals and $p$-adic heights on the two elliptic curves $E_1$ and $E_2$; in fact, quadratic Chabauty for the \emph{rational} points of $X$ is, essentially, a combination of quadratic Chabauty for the \emph{integral} points of $E_1$ and $E_2$ (see Remark \ref{rmk:QC_integral}). After Balakrishnan--Dogra's first examples of explicit quadratic Chabauty on a bielliptic curve \cite{BD18}, the first author \cite{Bianchi20} made some steps in the method more algorithmic and used this to provide further examples; an extension of this to number fields is given in \cite{QCnfs}.

In this article, we propose a simplification of the quadratic Chabauty function used in the computations of \cite{BD18, Bianchi20, QCnfs}: see Theorem \ref{thm:main} and Remark \ref{rmk:how_does_this_compare}. One of the resulting improvements is that, unlike in \cite{BD18,Bianchi20}, we can work with one function $\tilde{\rho}$ on the whole of $X(\Q_p)$, rather than having to consider different functions on two affine patches covering $X(\Q_p)$. This makes the computations less involved, and some of the algorithmic assumptions of \cite{BD18, Bianchi20} unnecessary. We give an elementary proof that the set of rational points is contained in 
\begin{equation*}
A = \{z\in X(\Q_p) : \tilde{\rho}(z)\in \Omega\},
\end{equation*}
where $\Omega$ is a finite subset of $\Q_p$ which we can describe explicitly in terms of our equations for $X$, $E_1$, $E_2$ and the reduction type of $E_1$ and $E_2$ at the primes of bad reduction. The proof relies on properties of global and local $p$-adic heights on elliptic curves. In particular, we use Mazur--Tate $p$-adic heights \cite{mazur-tate, MT91, MST, harvey}.

As a result, it is possible to run the algorithm on a database containing hundreds of curves, such as ours. To support our computations, we provide a precision analysis. The function $\tilde{\rho}$ is, locally, given by a power series $f(t)\in\Q_p[[t]]$, where $t$ is a local parameter. In order to compute the finite set of $p$-adic points $z$ satisfying $\tilde{\rho}(z)\in \Omega$, we need to have information about the $p$-adic valuation of the coefficients of $f(t)$. We derive lower bounds for these valuations in Section \ref{sec:prec_an}, where we also explain how to use this to compute the set $A$, up to some $p$-adic precision.

Finally, the set $A$ will in general be larger than $X(\Q)$. To complete our determination of the set $X(\Q)$, we apply the Mordell--Weil sieve as in \cite{BBM17} to exclude points in $A$ from belonging to $X(\Q)$, until we are left only with points in $A$ that we can recognise as points in $X(\Q)$. This normally requires computing the set $A$ for more than one prime $p$. See Section \ref{sec:computations} (in particular, \S \ref{subsec:step3}) for a description of a version of the Mordell--Weil sieve that is suitable to our setting, and that is based on \cite{BBM17}. For this step in our computations we use the code of Balakrishnan--Dogra--M\"uller--Tuitman--Vonk, available at \cite{QCMagma}.

For most of the curves in our database, we apply quadratic Chabauty together with the Mordell--Weil sieve to compute $X(\Q)$. For a minority of curves ($59$ of them), one of the two elliptic curves has rank $0$, so we can instead compute the rational points in a much more straightforward way by considering the preimages of the $\Q$-rational points on the rank $0$ elliptic curve quotient. Finally, for two curves, we suspected that $X(\Q) = \emptyset$, and we proved this using a sieve (without a preliminary quadratic Chabauty computation). See Section \ref{sec:computations} for the various steps in our computation. 

In summary, we have the following. Consider the set of genus $2$ curves defined over $\Q$ from \cite{genus2curvedatabase} (available at \cite{lmfdb}) which have points everywhere locally, are bielliptic over $\Q$, and whose Jacobians have rank $2$ over $\Q$. To this set add the quotient of the Shimura curve $X_0(10,19)$ by the Atkin-Lehner involution $w_{190}$ and the curve $y^2 = x^6 +6x^5 +39x^4 +52x^3 +39x^2 +6x+1$ (see Remark \ref{rmk:curves_of_interest} below). Let $\Delta$ be the resulting database of $413$ curves.

\begin{thm}
\label{thm:db}
The number of rational points of each curve in the database $\Delta$ is listed in \cite{RationalPts}.
\end{thm}

\begin{rmk}
For two of the curves in the database $\Delta$, the full set of rational points had been determined prior to our work, and is listed on \cite{lmfdb}. These curves are:
\begin{itemize}
    \item The quotient of the modular curve $X_0(129)$ by the group generated by the Atkin-Lehner involutions $w_3$ and $w_{43}$ (with LMFDB label \href{http://www.lmfdb.org/Genus2Curve/Q/5547/b/16641/1}{5547.b.16641.1}). The rational points for this curve were determined using $2$-cover descent by Bars--Gonz\'alez--Xarles \cite{Qcurves}, as well as using geometric quadratic Chabauty by Edixhoven--Lido \cite{Edixhoven_Lido}.
    \item The quotient of the modular curve $X_0(91)$ by its Fricke involution $w_{91}$ (with LMFDB label \href{http://www.lmfdb.org/Genus2Curve/Q/8281/a/8281/1}{8281.a.8281.1}). The $\Q(i)$-rational points on this curve were computed using quadratic Chabauty over number fields by Balakrishnan--Besser--M\"uller and the first author \cite{QCnfs}.
\end{itemize}
Our computations for these curves confirm the results of \cite{Qcurves, Edixhoven_Lido, QCnfs}.
\end{rmk}

\begin{rmk}\label{rmk:curves_of_interest}
The computation of the rational points of the following curves in $\Delta$ is of interest to prior work:
\begin{itemize}
    \item $X_0(10,19)/\langle w_{190} \rangle$: \cite[\S 3.1]{PS22};
    \item $y^2 = x^6 +6x^5 +39x^4 +52x^3 +39x^2 +6x+1$: \cite[Theorem 1 and Remark following]{LR22};
    \item $X_0(166)^*$ (with LMFDB label \href{https://www.lmfdb.org/Genus2Curve/Q/13778/a/27556/1}{13778.a.27556.1}): \cite[\S 2.4]{ACKP22}.
\end{itemize}
\end{rmk}

This paper is accompanied by our code on GitHub \cite{OurCode, RationalPts}. While the code for bielliptic quadratic Chabauty is an upgrade of the code pertaining to \cite{Bianchi20}, is written in \texttt{SageMath} \cite{sage} and is available at \cite{OurCode}, for the Mordell--Weil sieve computations we make extensive use of the \texttt{Magma} \cite{magma} code available at \cite{QCMagma}; the resulting implementation, as well as the results of our computations, can be found at \cite{RationalPts}.

\subsection*{Acknowledgements} It is a pleasure to thank Jennifer Balakrishnan, C\'eline Maistret and Steffen M\"uller for helpful discussions. We thank Jennifer Balakrishnan for proposing to us this project, based on a suggestion of Andrew Sutherland. We are grateful to Daniel Hast for private communication about \cite{Hast_2_cov}, and to David Roe for answering our questions about the LMFDB. We thank Jennifer Balakrishnan, Barinder Banwait, Raymond van Bommel, and Steffen M\"uller for useful comments on an earlier version of the paper. The first author is supported by an NWO Vidi grant. The second author is supported by NSF grant DMS-1945452 and Simons Foundation grant \#550023. 

\subsection{Notation}
Given a prime $q$, we denote by $\ordnop_q$ the $q$-adic valuation on $\Q_q$, normalised to be surjective onto $\Z$, and by $|\cdot|_q$ the standard absolute value on $\Q_q$, as well as its extension to $\overline{\Q_q}$.

\section{Quadratic Chabauty for genus 2 bielliptic curves}\label{sec:QC}
Let $X/\Q$ be a non-singular genus $2$ curve given by an equation of the form
\begin{equation}\label{eq:X}
X\colon y^2 = a_6x^6 + a_4x^4 + a_2 x^2 + a_0, \qquad a_i\in \Z
\end{equation}
and consider the elliptic curves
\begin{align}
E_1\colon y^2 = x^3 + a_4x^2 + a_2a_6x + a_0a_6^2 \label{eq:E1}\\
E_2 \colon y^2 = x^3 + a_2x^2 + a_4a_0x + a_6 a_0^2. \label{eq:E2}
\end{align}
There are degree $2$ maps $\varphi_i\colon X\to E_i$ given on affine points by
\begin{equation}\label{eq:vaphii}
    \varphi_1(x,y) = (a_6x^2, a_6 y), \qquad \varphi_2(x,y) = (a_0x^{-2}, a_0yx^{-3}).
\end{equation}
We denote by $\infty^{\pm}$ the two points at infinity in $X(\Q(\sqrt{a_6}))$ and by $\infty$ the point at infinity of an elliptic curve. 

Our goal is that of computing $X(\Q)$ under some assumptions on the ranks of $E_1$ and $E_2$. Since $\varphi_i(X(\Q))\subseteq E_i(\Q)$, the task is easy if one of the two elliptic curves is of rank $0$ over $\Q$.
The first interesting case arises when each of $E_1$ and $E_2$ has rank $1$ over $\Q$, and this is precisely the situation that we want to consider here.  So let us assume that 
\begin{equation*}
\rank(E_1(\Q)) = \rank(E_2(\Q)) = 1,
\end{equation*}
and let us fix a prime $p$ of good reduction for the model of $X$ given by \eqref{eq:X} (and hence also for \eqref{eq:E1} and \eqref{eq:E2}).

The strategy comprises two steps: first we compute a finite $p$-adic approximation of $X(\Q)$ inside $X(\Q_p)$ (quadratic Chabauty), and secondly we refine our approximation and extract the set $X(\Q)$ (Mordell--Weil sieve). Up until and including Section \ref{sec:prec_an}, our focus will be on the quadratic Chabauty part of the method, which is due to Balakrishnan--Dogra \cite{BD18}.

Some further examples and algorithmic observations and modifications to quadratic Chabauty for genus $2$ bielliptic curves were part of the first author's article \cite{Bianchi20}. Our starting point is the code \cite{codeQC_FB} provided with the latter article, which uses local $p$-adic height functions on $E_1$ and $E_2$ defined in terms of sigma functions \cite{MT91, MST, harvey} (differently from \cite{BD18}, which uses Coleman--Gross $p$-adic heights \cite{CG89}). 

We will need these to define a non-constant locally analytic function $\tilde{\rho}\colon X(\Q_p)\to \Q_p$ and a finite set $\Omega\subset \Q_p$ such that $\tilde{\rho}(X(\Q))\subseteq \Omega$; we further require that $\tilde{\rho}$ and $\Omega$ are computable, at least up to some desired $p$-adic precision. The set
\begin{equation}
\{z\in X(\Q_p) : \tilde{\rho}(z)\in \Omega\},
\end{equation}
computed to some $p$-adic precision, is our approximation of $X(\Q)$.

The main novelty compared to \cite{BD18, Bianchi20} is that, by keeping track of logarithmic singularities, we are able to work with a simpler function $\tilde{\rho}$: see Remark \ref{rmk:how_does_this_compare} below.

Let us now introduce what we need to define $\tilde{\rho}$ and $\Omega$. For now, we may assume more generally that $p$ is an odd prime. Further conditions on $p$ will be introduced only when needed.

First, we let $\log\colon\Z_p^{\times}\to \Q_p$ be the $p$-adic logarithm. It is possible to extend $\log$ to a function $\Q_p^{\times}\to \Q_p$ by choosing a value for $\log(p)$. It is customary (and natural in our situation \cite[Remark 2.1]{QCnfs}) to choose the Iwasawa branch, namely the one for which $\log(p) = 0$. Some of the summands of $\tilde{\rho}$ depend on this choice, so we fix this choice of branch. Note, however, that overall our $\tilde{\rho}$ will be branch-independent.

 Next, we consider the $p$-adic logarithm on an elliptic curve (we will eventually want to apply this to $E_1$ and $E_2$). Let $E$ be an elliptic curve over $\Q_p$ given by the Weierstrass equation
\begin{equation*}
E\colon y^2 = x^3 + A_2 x^2+A_4x+A_6, \qquad A_i \in \Z_p,
\end{equation*}
with point at infinity $\infty\in E(\Q_p)$.
The $p$-adic logarithm $\Log\colon E(\Q_p)\to \Q_p$ is the abelian group homomorphism defined as follows. 
For $P\in E(\Q_p)$, we let 
\begin{equation*}
\Log(P) = \int_{\infty}^{P} \omega, \qquad \text{where }\omega = \frac{dx}{2y},
\end{equation*}
and the integral is first defined by formal anti-differentiation in the formal group and then extended to $E(\Q_p)$ by insisting that the resulting function be a homomorphism. The map $\Log$ induces a homomorphism $E(\Q_p)\to H^0(E_{\Q_p}, \Omega^1)^{\vee}$, which is the $p$-adic Lie group logarithm \cite[III, \S 7.6]{Bourbaki_Lie}.  By \cite[Theorem 2.8]{Coleman_torsion}, if $E$ has good reduction, $\Log$ coincides with the Coleman integral of $\omega$ between $\infty$ and $P$. We will not use this, but we will use that if $P_1,P_2\in E(\Q_p)$ reduce to the same point modulo $p$, then 
\begin{equation*}
    \Log(P_1) - \Log(P_2) = \int_{P_2}^{P_1} \omega,
\end{equation*}
where the latter integral can also be computed by formal anti-differentiation of a local expansion of $\omega$. Since $\omega$ is holomorphic, $\Log$ is locally analytic, i.e.\ it can be expressed locally by a convergent power series (see \S \ref{subsec:prec_Log} for more details).
Moreover, $\Log$ vanishes at $P\in E(\Q_p)$ if and only if $P\in E(\Q_p)_{\tors}$ (see e.g.\ \cite[IV, Theorem 6.4\thinspace{}(b)]{silverman_AEC}, or, more generally,
\cite[Proposition 3.1]{Coleman_torsion}).
In summary, $\Log\colon E(\Q_p)\to\Q_p$ is a locally analytic group homomorphism, whose kernel is $E(\Q_p)_{\tors}$.

Finally, the most technical ingredient that we need is the theory of $p$-adic heights on elliptic curves. The appearance of $p$-adic heights in the definition of $\tilde{\rho}$ partly justifies the adjective ``quadratic'' in the name of the method. Assume now that our elliptic curve is defined over $\Q$:
\begin{equation*}
E\colon y^2 = x^3 + A_2 x^2+A_4x+A_6, \qquad A_i \in \Z
\end{equation*}
and has good reduction at $p$. Fix a constant $c\in \Q_p$, or, equivalently, a differential $\eta$ of the form $\eta = (x+c)\frac{dx}{2y}$. Note that the class of $\eta$ spans a one-dimensional subspace of $H^1_{\dR}(E/\Q_p)$ complementary to the space of holomorphic forms, and conversely every such subspace is spanned by the class of a differential of that form. 

There are several theories of $p$-adic heights in the literature, although many comparison results are now known. Here we use the same height as in \cite{Bianchi20}.  This is essentially the one of Mazur--Stein--Tate\cite{MST} (at least when the reduction is ordinary and we pick the above subspace to be the unit root eigenspace of Frobenius, as we will do in Section \ref{sec:prec_an}), but we further consider its decomposition into a sum of $p$-adic local N\'eron functions at every prime $q$
\begin{equation*}
\lambda_q\colon E(\Q_q)\setminus\{\infty\}\to \Q_p
\end{equation*}
as in \cite[\S 2A, 4A]{Bianchi20}. At the primes different from our working prime $p$, these are obtained from real-valued N\'eron functions with respect to the divisor $2(\infty)$ \cite[Chapter VI]{silvermanadvancedtopics} by replacing the real logarithm with the $p$-adic one, and therefore satisfy similar properties. 
At the prime $p$, the local height $\lambda_p$ depends on the choice of $c$; in the formal group of $E$ at $p$, it also depends on the branch of the $p$-adic logarithm.

The properties of these local height functions that we need for the main theorem are listed in Propositions \ref{prop:prop_lambdap} and \ref{prop:prop_lambdaq}, and more details on $\lambda_p$ are also provided in the precision analysis (\S \ref{subsec:prec_lambda}). Roughly, the local function $\lambda_q$ is well-behaved on the subset of $E(\Q_q)$ consisting of points with coordinates in $\Z_q$, where by ``well-behaved'' we mean that it has finite image if $q\neq p$ and is locally analytic if $q=p$. In the subset of $E(\Q_q)$ consisting of points reducing to $\infty$, the function $\lambda_q$ has a logarithmic term if $q=p$ and takes infinitely many values if $q\neq p$.  When we consider the problem of computing $X(\Q)$, we are able to control this unboundedness by using both maps $\varphi_1$ and $\varphi_2$ defined in \eqref{eq:vaphii} (see Theorem \ref{thm:main}\thinspace{}\ref{thm:main_b}).

\begin{prop}\label{prop:prop_lambdap}
\leavevmode
\begin{enumerate}[label=(\alph*)]
    \item\label{prop:prop_lambdap_loc_an} The local N\'eron function $\lambda_p$ is locally analytic away from the residue disc of the point at infinity.
    \item\label{prop:prop_lambdap_log_term} Let $t=-\frac{x}{y}$. Then, the expansion of $\lambda_p$ in the disc of $\infty$ in terms of $t$ is of the form $-2\log(t)+ O(t)$. 
    \end{enumerate}
    \end{prop}
    \begin{proof}
    See \cite{Bianchi20} and \S \ref{subsec:prec_lambda}.
    \end{proof}
    \begin{prop}\label{prop:prop_lambdaq} Let $q\neq p$.
    \begin{enumerate}[label=(\alph*)]
    \item\label{prop:prop_lambdaq_nonsign} If $P\in E(\Q_q)$ reduces to a non-singular point modulo $q$, then 
    \begin{equation*}
        \lambda_q(P) = \log(\max\{1,|x(P)|_q\}).
    \end{equation*}
    \item \label{prop:prop_lambdaq_int}Let $W_q^E$ be the set of values attained by $\lambda_q$ on points in $E(\Q_q)$ of the form $(x,y)$ with $x,y\in \Z_q$. Then $W_q^E$ is finite, explicitly computable, and $\{0\}$ for all but finitely many $q$ (in particular, $W_q^E\subseteq \{0\}$ at all primes of good reduction for the given model for $E$).
    \end{enumerate}
\end{prop}
\begin{proof}
For part (a), see \cite[Lemma 2.1]{Bianchi20}.  For part (b), see \cite[Lemma 6.4]{Bianchi20}, which as stated seems to be specific to $E_i$, but holds more generally. It provides an explicit description of $W_q^E$. 
\end{proof}

For a global point $P\in E(\Q)$, the global $p$-adic height is then 
\begin{equation}\label{eq:hp}
h_p(P) =\begin{cases}
\sum_{q} \lambda_q(P) \quad & \text{if } P\neq \infty,\\
0 \quad & \text{otherwise.}
\end{cases}
\end{equation}
By Proposition \ref{prop:prop_lambdaq}, this is well-defined, as all but finitely many of the $\lambda_q(P)$ are equal to $0$, for a given $P$. The crucial property satisfied by $h_p$ that we need is the following:
\begin{equation}\label{eq:hp_quad}
h_p(mP) = m^2 h_p(P), \qquad \text{for all } m\in \Z, P\in E(\Q).
\end{equation}
We are now ready to introduce the quadratic Chabauty function $\tilde{\rho}$ mentioned above, and the set $\Omega$.  See Remark \ref{rmk:how_does_this_compare} for an explanation of how this differs from the explicit quadratic Chabauty function used in \cite{BD18, Bianchi20} (and in the generalisation to number fields of \cite{QCnfs}).

For a prime $q$, let 
\begin{equation*}
Z_q = X(\Q_q) \setminus \{P: x(P)\in \{0,\infty\}\}.
\end{equation*}

\begin{thm}\label{thm:main}
Suppose that each of $E_1$ and $E_2$ has rank $1$ over $\Q$, and let $p$ be a prime of good reduction for the equation \eqref{eq:X} for $X$. For each $i\in\{1,2\}$, fix a choice of subspace of $H^1_{\dR}(E_i/\Q_p)$ complementary to the space of holomorphic forms, and consider the corresponding global height $h_p$ and local N\'eron functions $\lambda_q$, at every $q$. Let $P_i\in E_i(\Q)$ be a point of infinite order and let
\begin{equation*}
\alpha_i = \frac{h_p(P_i)}{\Log^2(P_i)}.
\end{equation*}
Then:
\begin{enumerate}[label=(\alph*)]
\item \label{thm:main_a} The constant $\alpha_i$ is independent of the choice of $P_i$.
\item \label{thm:main_b} The function $\rho\colon Z_p\to \Q_p$ given by
\begin{equation*}
\rho(z) =     \lambda_p(\varphi_1(z)) - \lambda_p(\varphi_2(z)) - 2\log(x(z)) - \alpha_1 \Log^2(\varphi_1(z)) + \alpha_2\Log^2(\varphi_2(z)) 
\end{equation*}
can be continued to a locally analytic function $\tilde{\rho}\colon X(\Q_p)\to \Q_p$.
\item  For a prime $q\neq p$, let
\begin{align*}
    \Omega_q = (-W_q^{E_1} + W_q^{E_2} + \{-n\log q : -\ordnop_q(a_6) \leq n \leq \ordnop_q(a_0) \text{ and } n \equiv 0 \bmod{2}\})\\
    \cup (\log|a_0|_q-W_q^{E_1} ) \cup (- \log|a_6|_q + W_q^{E_2}).
\end{align*}
and set
\begin{align*}
\Omega = \biggl\{\sum_{q\: \text{bad}} w_q:w_q\in  \Omega_q \biggr\},
\end{align*}
where the sum runs over all primes at which $X$ has bad reduction. Then $\Omega$ is finite and 
\begin{equation*}
\tilde{\rho}(X(\Q)) \subseteq \Omega.
\end{equation*}
\end{enumerate}
\end{thm}

\begin{rmk}
We explain the notation in Part (c) of the theorem. The finite sets $W_q^{E_i}\subset \Q_p$ are defined in Proposition \ref{prop:prop_lambdaq}. Given $A,B\subset\Q_p$ and $a\in \Q_p$, we write:
\begin{equation*}
A+B = \{a+b:a\in A, b\in B\}, \qquad -A = \{-a:a\in A\}, \qquad a+B = \{a\}+B.
\end{equation*}
\end{rmk}

\begin{rmk}\label{rmk:how_does_this_compare}
By Proposition \ref{prop:prop_lambdap}\thinspace{}\ref{prop:prop_lambdap_log_term}, $\lambda_p(\varphi_1(z))$ has a logarithmic term around points at infinity, $\lambda_p(\varphi_2(z))$ has a logarithmic term around points with zero $x$-coordinate. The term $\log(x(z))$ also has a logarithmic term around each of these points. To make up for this, in \cite[Corollary 8.1]{BD18} and \cite [Proposition 6.5]{Bianchi20} one considered two different affine patches covering $X(\Q_p)$, and correspondingly applied suitable translations on the elliptic curves to move away from these problematic discs. What we propose to do here instead is to discard the logarithmic terms at once, since these overall cancel out. 

Moreover, unlike in \cite[Corollary 8.1]{BD18} and \cite [Proposition 6.5]{Bianchi20}, we do not assume that $a_6 = 1$. Similarly to \cite{Bianchi20}, we give here an elementary proof of the resulting quadratic Chabauty criterion, which does not require any $p$-adic Hodge theory. As a result, we obtain the explicit description of $\Omega$ provided in the theorem statement. Using results in \cite{BD18}, a smaller $\Omega$ may sometimes be chosen if there are some primes $q$ at which $X$ has bad but potentially good reduction: see Proposion \ref{prop:potential} below.
\end{rmk}

\begin{proof}[Proof of Theorem \ref{thm:main}]
\item[(a)]  First note that $\alpha_i$ is well-defined since $\Log$ vanishes on torsion points only. The independence of $\alpha_i$ on the choice of $P_i$ is a standard argument in quadratic Chabauty methods. Namely, first note that since $\Log$ is a homomorphism, its square satisfies
\begin{equation*}
\Log^2(mP) = m^2\Log^2(P) \qquad \text{for all } m\in\Z\quad \text{and}\quad P\in E(\Q).
\end{equation*}
Since $\Log^2$ is also non-degenerate and $\rank(E_i(\Q)) = 1$, any other function $E(\Q)\to\Q_p$ that transforms quadratically with respect to multiplication by $m$ on the elliptic curve must be a scalar multiple of $\Log^2$. This is in particular the case for $h_p$, in view of \eqref{eq:hp_quad}.

\item[(b)] Each term in $\rho$ is locally analytic, with the following exceptions: the function $\lambda_p(\varphi_1(z))$ has a logarithmic term in residue discs of points with $x(z) = \infty$,  the function $\lambda_p(\varphi_2(z))$ has a logarithmic term in residue discs of points with $x(z) = 0$, and $\log(x(z))$ has a logarithmic term in residue discs of points with $x(z) \in\{0,\infty\}$. Using Proposition \ref{prop:prop_lambdap}\thinspace{}\ref{prop:prop_lambdap_log_term}, we see that, overall, the logarithmic terms add up to $0$. For more details, see the proof of part (c) or Section \ref{sec:prec_an}.
\item[(c)] First note that $\Omega$ is finite, since $\Omega_q$ is finite and there are finitely many primes of bad reduction for $X$. Also note that
\begin{equation*}
   \Omega \supseteq \Omega^{\prime} \colonequals  \biggl\{\sum_{q\neq p} w_q:w_q\in  \Omega_q \biggr\},
\end{equation*}
since $\Omega_q\subseteq \{0\}$ if $q$ is a prime of good reduction (by Proposition \ref{prop:prop_lambdaq}\thinspace{}\ref{prop:prop_lambdaq_int}, the fact that the primes dividing $a_0$ or $a_6$ are of bad reduction for the given equation for $X$, and that the primes of good reduction for $X$ are of good reduction for the equations \eqref{eq:E1} and \eqref{eq:E2} as well). Therefore, it suffices to show that $\tilde{\rho}(X(\Q))\subseteq \Omega^{\prime}$. 

Let us first assume that $z\in X(\Q)\cap Z_p$. Then, by \eqref{eq:hp} and part \ref{thm:main_a},
\begin{equation*}
   \tilde{\rho}(z) =  \rho(z) = 
    \sum_{q\neq p}\left(-\lambda_q(\varphi_1(z)) + \lambda_q(\varphi_2(z))+2\log |x(z)|_q\right). \end{equation*}
    For $z_q\in  Z_q$, define
    \begin{equation*}
        w_q(z_q) = -\lambda_q(\varphi_1(z_q)) + \lambda_q(\varphi_2(z_q))+2\log |x(z_q)|_q.
    \end{equation*}
    Thus, for $z$ as above, we have
    \begin{equation*}
       \rho(z)\in \Omega^{\prime\prime} \colonequals \biggl\{ \sum_{q\neq p} w_q(z_q) : (z_q) \in \prod_{q\neq p}  Z_q\biggr\}.
    \end{equation*}
    The following case distinction (which uses Proposition \ref{prop:prop_lambdaq}) shows that $\Omega^{\prime\prime}\subseteq \Omega^{\prime}$:   \begin{enumerate}
        \item If $-\ordnop_q(a_6)\leq 2\ordnop_q(x(z_q))\leq \ordnop_q(a_0)$, both $\varphi_1(z_q)$ and $\varphi_2(z_q)$ are integral and 
        \begin{equation*}
            w_q(z_q) \in -W_q^{E_1} + W_q^{E_2} + \{-n\log q : -\ordnop_q(a_6) \leq n \leq \ordnop_q(a_0) \text{ and } n \equiv 0 \bmod{2}\}.
        \end{equation*}
        \item If $2\ordnop_q(x(z_q))> \ordnop_q(a_0)$, then $\varphi_1(z_q)$ is integral, $\varphi_2(z_q)$ is not. We have
        \begin{equation*}
            w_q(z_q) \in -W_q^{E_1} + \log|a_0x(z_q)^{-2}|_q+ 2\log |x(z_q)|_q = -W_q^{E_1} + \log|a_0|_q.
        \end{equation*}
        \item If $2\ordnop_q(x(z_q)) <-\ordnop_q(a_6)$, then $\varphi_2(z_q)$ is integral, while $\varphi_1(z_q)$ is not. We have
        \begin{equation*}
            w_q(z_q) \in W_q^{E_2} - \log|a_6x(z_q)^2|_q + 2\log|x(z_q)|_q = - \log|a_6|_q + W_q^{E_2}.
        \end{equation*}
        \end{enumerate}
        Finally we need to compute $\tilde{\rho}(z)$ for $z\in X(\Q)$, $x(z)\in\{0,\infty\}$ (if such a point exists). 
        
        If $z\in X(\Q)$ with $x(z) = 0$, then
    \begin{equation*}
    \lambda_p(\varphi_1(z)) - \alpha_1\Log^2(\varphi_1(z)) + \alpha_2\Log^2(\varphi_2(z)) = -\sum_{q\neq p} \lambda_q(\varphi_1(z)) \in \biggl\{\sum_{q\: \text{bad}} w_q: w_q\in -W_q^{E_1}\biggr\},
    \end{equation*}
    since $\varphi_1(z)$ is an integral point and $\varphi_2(z)$ is the point at infinity. The remaining summands in $\rho$ are not individually well-defined at $z$ and we circumvent this issue as explained in part (b), by expanding around $z$. For instance,  a parametrisation of the disc containing $z$ is given by 
    \begin{equation*}
        z(t) = (x(t),y(t)) = (t,  \sqrt{a_0}+O(t)),
    \end{equation*}
where $\sqrt{a_0}$ is a suitably chosen square root of $a_0$. Thus, 
    \begin{equation*}
    \varphi_2(x(t),y(t)) = (a_0t^{-2}, a_0\sqrt{a_0}t^{-3}+O(t^{-2}))
    \end{equation*}
    and so, by Proposition \ref{prop:prop_lambdap}\thinspace{}\ref{prop:prop_lambdap_log_term}, 
    \begin{equation*}
     -\lambda_p\circ\varphi_2(z(t))-2\log\circ\, x(z(t)) = 2\log\left(\frac{1}{\sqrt{a_0}}t + O(t^2)\right) -2\log(t) = -\log(a_0) + O(t).
    \end{equation*}
    We conclude that, if $a_0$ is a square in $\Q$, then
    \begin{equation*}
    \tilde{\rho}(0,\sqrt{a_0}) =-\sum_{q\neq p}\lambda_q\circ\varphi_1(0,\sqrt{a_0}) - \log(a_0)=\sum_{q\neq p}(-\lambda_q\circ \varphi_1 (0,\sqrt{a_0}) + \log|a_0|_q)\in \Omega.
    \end{equation*}
    
    Similarly,  if $z=\infty^{\pm}\in X(\Q)$, then $\varphi_2(z)$ is an integral point on $E_2$ and $\varphi_1(z)$ is the point at infinity on $E_1$. Thus, 
\begin{equation*}
    -\lambda_p(\varphi_2(z))-\alpha_1\Log^2(\varphi_1(z))+\alpha_2\Log^2(\varphi_2(z)) = \sum_{q\neq p}\lambda_q(\varphi_2(z))\in \biggl\{\sum_{q\: \text{bad}}w_q: w_q\in W_q^{E_2}\biggr\}.
\end{equation*}
As for the remaining term $\lambda_p(\varphi_1(z)) -2\log(x(z))$, we have the following. A parametrisation around $z$ is given by
\begin{equation*}
    z(t) = (x(t),y(t)) = (t^{-1},  \sqrt{a_6}t^{-3} + O(t^{-2})).
\end{equation*}
 Hence
\begin{equation*}
\varphi_1(z(t)) = (a_6 t^{-2}, a_6\sqrt{a_6}t^{-3} + O(t^{-2})).
\end{equation*}
By Proposition \ref{prop:prop_lambdap}\thinspace{}\ref{prop:prop_lambdap_log_term}, 
\begin{equation*}
\lambda_p\circ \varphi_1(z(t)) -2\log\circ\, x(z(t)) = -2\log\left(-\frac{1}{\sqrt{a_6}}t + O(t^2)\right) -2\log(t^{-1})  = \log(a_6) + O(t).
\end{equation*}
So if $a_6$ is a square in $\Q$, then
\begin{equation*}
\rho(\infty^{\pm}) = \sum_{q\neq p} (\lambda_q\circ\varphi_2(\infty^{\pm})-\log|a_6|_q)\in \Omega.
\end{equation*}
\end{proof}

As in the proof of Theorem \ref{thm:main}, for $z_q\in Z_q$, let 
\begin{equation*}
w_q(z_q) = -\lambda_q(\varphi_1(z_q)) + \lambda_q(\varphi_2(z_q))+2\log |x(z_q)|_q.
\end{equation*}
Otherwise, if $z_q\in X(\Q_q)\setminus Z_q$, we set
\begin{equation*}
    w_q(z_q) = \begin{cases}
    -\lambda_q(\varphi_1(z_q)) + \log|a_0|_q\quad & \text{if } x(z_q) = 0\\
    \lambda_q(\varphi_2(z_q)) -\log|a_6|_q\quad & \text{if } x(z_q) = \infty.
    \end{cases}
\end{equation*}

\begin{prop} \label{prop:potential}
If $q$ is a prime of potential good reduction for $X$, then $w_q$ is constant on $X(\Q_q)$. Therefore, in Theorem \ref{thm:main} we may replace $\Omega_q$ with $\{w_q(z_q)\}$, where $z_q$ is any point in $X(\Q_q)$.
\end{prop}
\begin{proof}
On $Z_q$, this follows from \cite[Lemma 5.4]{BD18} and a computation analogous to that of \cite[Lemma 7.7]{BD18} (where $X$ is assumed monic). We extend to $z_q\not\in Z_q$ using continuity properties of local heights.
\end{proof}

As a first step for the determination of $X(\Q)$, we would like to be able to compute the set
\begin{equation*}
A = \{z\in X(\Q_p) : \tilde{\rho}(z)\in \Omega\}\supseteq X(\Q).
\end{equation*}
In the next section we explain how to do so up to some finite $p$-adic precision. 

\section{Precision analysis}\label{sec:prec_an}
\subsection{Preliminaries}\label{subsec:prec_intro}
Let $p$ be a prime of good reduction for the model of $X$, and hence for the given models for $E_1$ and $E_2$. In particular, $p \nmid a_0a_6$.

In Theorem \ref{thm:main}, we considered the locally analytic function $\tilde{\rho}\colon X(\Q_p)\to \Q_p$ obtained by extending $\rho\colon Z_p\to \Q_p$ defined by
\begin{equation}\label{eq:rho}
\rho(z) =     \lambda_p(\varphi_1(z)) - \lambda_p(\varphi_2(z)) - 2\log(x(z)) - \alpha_1 \Log^2(\varphi_1(z)) + \alpha_2\Log^2(\varphi_2(z)).
\end{equation}

Since $\tilde{\rho}$ is locally analytic, it can be expanded in each residue disc $D$ of $X(\Q_p)$ as a power series in $\Q_p[[t]]$, where $t$ is a local coordinate at a fixed $z\in D$ (that is, $t$ is a uniformiser at $z$ that reduces to a uniformiser for $\overline{z}\in X(\F_p)$). Our goal here is to find lower bounds for the $p$-adic valuation of the coefficients of these series, which will allow us to deduce information about the solutions to $\tilde{\rho}(z) - w$, for $w\in \Omega$, in the disc $D$.

For a similar precision analysis in the simpler setting of the classical method of Chabauty and Coleman, see also \cite[Section 3]{BBCE19}.

Given a residue disc $D$ in $X(\Q_p)$, there are two choices to make. First, we need to pick a point $z\in D$ and, secondly, we need to pick a local coordinate at $z$.  Let $\overline{z}\in X(\F_p)$ be the image of $D$ under reduction.

We proceed as follows:
\begin{enumerate}[label=(\roman*)]
    \item\label{it:affine} if $\overline{z}$ is affine with $y(\overline{z})\neq 0$, we take $z\in X(\Q_p)$ to be the unique point satisfying $x(z)\in \Z$, $0\leq x(z)\leq p-1$, $\overline{x(z)} = x(\overline{z})$ and $\overline{y(z)} = y(\overline{z})$. As for the local coordinate, we take $t = x - x(z)$. Note that then $y(t)\in \Z_p[[t]]$ is the unique solution to $y(t)^2 = F(x(t))$ such that $y(0) = y(z)$. 
    \item if $\overline{z}$ is Weierstrass (i.e.\ $y(\overline{z}) = 0$), we take $z\in X(\Q_p)$ to be the unique Weierstrass point in the disc. As for the local coordinate, choose $t = y$. Then $x(t)\in \Z_p[[t]]$ is the unique solution to $F(x(t)) = y(t)^2$ such that $x(0) = x(z)$.
    \item\label{it:infty} if $\overline{z}$ is a point at infinity, we take $z\in X(\Q_p)$ to be the unique point at infinity in the disc. For the local coordinate, take $t = x^{-1}$. Then $y(t) = \sqrt{a_6}t^{-3} + O(t^{-2})$, for a suitable choice of $\sqrt{a_6}$.
\end{enumerate}
\begin{rmk}
The existence of such a point $z\in X(\Q_p)$ and such a parametrisation in \ref{it:affine} -- \ref{it:infty} is guaranteed by the good reduction assumption and by Hensel's lemma.
\end{rmk}

We analyse each term in \eqref{eq:rho} separately and then draw conclusions at the end. In fact, for the terms that involve images of $z$ under $\varphi_1$ or $\varphi_2$ (all of them, except for $\log(x(z))$), as a preliminary step we ignore that these come from points on $X$ and analyse the corresponding terms on a generic elliptic curve. Note that, as we already saw in the proof of Theorem \ref{thm:main}, in the disc of a point at infinity or a point with $x$-coordinate reducing to $0$ modulo $p$, not every term is individually expressible as a power series. 

Recall that $\log$ denotes our chosen branch of the $p$-adic logarithm. We will write $\log_p$ for the real logarithm with respect to base $p$.

We start with an auxiliary lemma.

\begin{lemma}\label{lemma:auxiliary}
Let $f(T) = \sum_{n=0}^{\infty} B_n T^n\in \Q_p[[T]]$ with
\begin{equation*}
\ord_p(B_n) \geq - \ord_p(n) + \alpha(n) \qquad \text{for all}\ n\geq 1,
\end{equation*}
where $\alpha(n)$ is a (not necessarily strictly) decreasing function of $n$. Let $g(t) = \sum_{i=1}^{\infty} b_i t^i$ for some $b_i\in \Z_p$. Then
\begin{equation*}
f(g(t)) = B_0 + \sum_{n=1}^{\infty} C_n t^n,
\end{equation*}
with $\ord_p(C_n)\geq - \ord_p(n) + \alpha(n)$ for all $n\geq 1$.
\end{lemma}
\begin{proof}
For $n\geq 1$, we have
\begin{align*}
C_n &= \sum_{m=1}^{n} B_m\cdot\left(\text{coefficient of } t^n \text{ in } \biggl(\sum_{i=1}^n b_i t^i\biggr)^m\right)\\
&= \sum_{m=1}^n B_m\cdot\left(\text{coefficient of } t^n \text{ in } \biggl(\sum_{i_1+\cdots + i_n =m}\binom{m}{i_1,\ldots,i_n} \prod_{k=1}^n (b_k t^k)^{i_k}\biggr)\right).
\end{align*}
Thus,
\begin{equation*}
    \ord_p(C_n) \geq \min\left\{\ord_p\binom{m}{i_1,\ldots, i_n} +\ord_p(B_m)\right\}.
\end{equation*}
where the minimum is taken over all $m\leq n$ and $i_1,\dots, i_n\geq 0$ satisfying $\sum_{k=1}^n i_k= m$ and $\sum_{k=1}^n ki_k = n$.

For such $i_1,\dots, i_n$, there must exist $k\in \{1,\dots,n\}$ for which $\ord_p(i_k) \leq \ord_p(n)$.  Then we have
\begin{equation*}
\binom{m}{i_1,\ldots,i_n} = \frac{m}{i_k} \binom{m-1}{i_1,\dots, i_{k-1},i_k -1,i_{k+1}, \dots, i_n}.
\end{equation*}
Therefore, 
\begin{align*}
    \ord_p(C_n)&\geq \min_{m\leq n}\{\ord_p(m) -\ord_p(n) + \ord_p(B_m) \}\\
    &\geq \min_{m\leq n}\{ -\ord_p(n) +\alpha(m) \}\geq -\ord_p(n) + \alpha(n),
\end{align*}
since $\alpha$ is a decreasing function.
\end{proof}

\begin{rmk}
In the special case where $f(T) = \log(1+T)$ (so $\alpha(n) = 0$ for all $n$), one could recover the same precision estimate in a more straightforward way:
\begin{equation*}
    f(g(t)) = \log(1+g(t)) = \int \frac{g^{\prime}(t)}{1+g(t)} dt.
\end{equation*}
\end{rmk}

\subsection{Precision of $\log(x(t))$-term}\label{subsec:prec_log}

\begin{lemma}\label{lemma:log}
Let $D$ be a residue disc of $X(\Q_p)$ and choose $z\in D$ and a local coordinate at $z$ as in \S \ref{subsec:prec_intro}.
\begin{enumerate}
    \item If $z\in \{\infty^{\pm}\}$, then $\log(x(t)) = -\log(t)$;
    \item If $x(z) = 0$, then $\log(x(t)) = \log(t)$;
    \item Otherwise,
    \begin{equation*}
        \log(x(t)) = \sum_{n=0}^{\infty} C_n t^n\in \Q_p[[t]],
        \end{equation*}
        with $\ord_p(C_0)\geq 1$ and $\ord_p(C_n)\geq -\ord_p(n)$
        for all $n\geq 1$.
\end{enumerate}
\end{lemma}
\begin{proof}
The first two cases are trivial. For the remaining case, we have
\begin{equation*}
x(t) = \alpha\biggl(1 + \sum_{i=1}^{\infty} b_i t^i\biggr), \qquad \text{for some } b_i\in \Z_p, \alpha\in \Z_p^{\times}.
\end{equation*}
Therefore,
\begin{equation*}
    \log(x(t)) = \log(\alpha) + \log\biggl(1+  \sum_{i=1}^{\infty} b_i t^i\biggr) = \log(\alpha) + O(t),
\end{equation*}
which shows the claim on the constant term. In order to bound the valuation of the other terms, we apply Lemma \ref{lemma:auxiliary} to $f(T) = \sum_{n=1}^{\infty}\frac{(-1)^{n+1}}{n} T^n$ and $g(t) = \sum_{i=1}^{\infty} b_i t^i$. 
\end{proof}

\subsection{Precision of $\Log$-terms}\label{subsec:prec_Log}
Let $E$ be an elliptic curve over $\Q_p$ given by the Weierstrass equation
\begin{equation*}
E : y^2 = x^3 + A_2 x^2+A_4x+A_6, \qquad A_i \in \Z_p.
\end{equation*}
Recall that, for $P\in E(\Q_p)$, we have 
\begin{equation*}
\Log(P) = \int_{\infty}^{P} \omega, \qquad \text{where }\omega = \frac{dx}{2y},
\end{equation*}
and the integral is first defined by formal anti-differentiation in the formal group of $E$ at $p$ and then extended to $E(\Q_p)$ by linearity.

In order to obtain the expansion of $\Log$ in a local coordinate $T$ for a residue disc, we may break up the path of integration: given $P_0\in E(\Q_p)$ reducing to $P$ modulo $p$ and with $T(P_0)=0$, we have
\begin{equation*}
\Log(P) = \Log(P_0) + \int_{P_0}^P \omega,
\end{equation*}
where the latter integral can be computed by expanding $\omega$ as a power series in $T$, formally integrating and evaluating at $T(P)$.
\begin{lemma}\label{lemma:Log}
Let $T$ be a local coordinate for an arbitrary point $P_0$ in a residue disc of $E(\Q_p)$. Then
\begin{equation*}
\Log^2(P(T)) = \sum_{n=0}^{\infty} C_nT^n,
\end{equation*}
where $\ord_p(C_0),\ord_p(C_1) \geq 0$ and $\ord_p(C_n)\geq -\lfloor{\log_p(n-1)\rfloor}-\ord_p(n)$ for $n\geq 2$.
Moreover, if $p\nmid \#E(\F_p)$, then $\ord_p(C_0) \geq 2$ and $\ord_p(C_1)\geq 1$.
\end{lemma}
\begin{proof}
We are interested in finding lower bounds for the valuation of the coefficients of 
\begin{equation*}
\Log^2(P(T)) = \Log^2(P_0)  + 2\Log(P_0)\int_{P_0}^{P(T)} \omega (T) + \left(\int_{P_0}^{P(T)} \omega (T)\right)^2. \end{equation*} 
Since $\omega$ is holomorphic and non-vanishing and $T$ reduces to a uniformiser modulo $p$, we have 
\begin{equation*}
\omega(T) = f(T) dT \qquad \text{for some } f(T)\in \Z_p[[T]]^{\times}.
\end{equation*}
Therefore,
\begin{equation}\label{eq:int_omega}
    \int_{P_0}^{P(T)}\omega(T) =  c_0T + \frac{c_1}{2}T^2 + \frac{c_2}{3}T^3 + \cdots ,\qquad \text{for some}\quad c_i\in \Z_p, c_0\in \Z_p^{\times}.
\end{equation}
Let $m>0$ be the smallest integer such that $mP_0$ belongs to the disc at infinity. Since $m$ is a divisor of $\#E(\F_p)$, by the Hasse bound \cite[V, Theorem 1.1]{silverman_AEC} we know that $0\leq \ord_p(m)\leq 1$. We have
\begin{equation*}
\Log(P_0) = \frac{\Log(mP_0)}{m}.
\end{equation*}
Now, by \eqref{eq:int_omega} with $P_0 = \infty$ we see that $\ord_p(\Log(mP_0))\geq 1$ and hence, by the above considerations on the valuation of $m$, $\ord_p(\Log(P_0))\geq 1-\ord_p(\#E(\F_p))\geq 0$. The statements about the constant and linear coefficient follow.

By \eqref{eq:int_omega}, the $n$-th coefficient in the $T$-expansion of the integral of $\omega(T)$ has valuation at least $-\ord_p(n)$. As for
\begin{align*}
\left(\int_{P_0}^{P(T)} \omega (T)\right)^2 = \left(c_0T + \frac{c_1}{2}T^2 + \frac{c_2}{3}T^3 + \cdots\right)^2 =  \sum_{n = 2}^{\infty} \alpha_n T^n
\end{align*}
we have
\begin{equation*}
\ord_p(\alpha_n) \ge  \min_{\substack{i+j = n\\
i,j\geq 1}}\ord_p\left(\frac{ 1}{ij}\right) \geq -\lfloor{\log_p(n-1)\rfloor}  -\ord_p(n),
\end{equation*}
since $\ord_p(i), \ord_p(j)\leq \lfloor{\log_p(n-1)\rfloor}$, but also $\min\{\ord_p(i),\ord_p(j)\}\leq \ord_p(n)$.
\end{proof}

\subsection{Precision of $\lambda_p$-terms} \label{subsec:prec_lambda}
Let $E$ be as in \S \ref{subsec:prec_Log}.
We are ultimately interested in applying the considerations of this subsection to the case where $E$ is $E_1$ or $E_2$.
To simplify our task (or, in fact, to obtain better bounds; see below), we make the following 

\begin{assumption}
\label{ass:ordinary}
The prime $p$ is of ordinary reduction for $E_1$ and $E_2$. 
\end{assumption}

Thus we assume here that $p$ is a prime of good ordinary reduction for $E$, so we can and shall work with the canonical local height at $p$. This is defined in terms of the canonical $p$-adic sigma function\footnote{If we do not assume that $p$ is of ordinary reduction for $E$, we can replace the Mazur--Tate $p$-adic sigma function with some other $p$-adic sigma function $\sigma_p^{\prime}(T) = T + O(T^2)\in \Q_p[[T]]$; for example, Bernardi's \cite{bernardi}. The precision bounds then need to be modified appropriately, since $\sigma_p^{\prime}(T)$ does not, in general, have $p$-adically integral coefficients.}
\begin{equation*}
\sigma_p(T) = T + O(T^2)\in \Z_p[[T]]
\end{equation*}
of Mazur--Tate \cite{MT91}, as follows. We view $\sigma_p$ as a function on points of $E(\Q_p)$ in the formal group $E_f$ by setting
\begin{equation*}
\sigma_p(P) \colonequals \sigma_p\left(-\frac{x(P)}{y(P)}\right),\qquad \text{for}\ P\in E_f(\Q_p).
\end{equation*}

If $P\in E_f(\Q_p)\setminus \{O\}$, its canonical local height at $p$ is then given by
\begin{equation}\label{eq:lambadp_inf}
\lambda_p(P) = -2\log(\sigma_p(P)).
\end{equation}
Note that, in this case, $\lambda_p(P)$ depends on the choice of a branch of the $p$-adic logarithm and general theory on heights suggests that we should choose the branch of the logarithm which is trivial at $p$ (cf.\ \cite[Remark 2.1]{QCnfs}). In practice, for our applications to computing $\tilde{\rho}$ this will not matter, in view of the global cancellation of the logarithmic terms.

If $P$ is non-torsion and does not reduce to the point at infinity, then let $m$ be a positive integer such that $mP\in E_f(\Q_p)$. If we choose the smallest such $m$, then $m$ divides $\#E(\F_p)$.  We define 
\begin{equation}\label{eq:lambdapaff}
\lambda_p(P) = -\frac{2}{m^2}\log\left(\frac{\sigma_p(mP)}{\phi_m(P)}\right),
\end{equation}
where $\phi_m\in \Z[A_2,A_4,A_6][x,y]$ is the $m$-th division polynomial \cite[Exercise 3.7]{silverman_AEC}. This is characterised uniquely up to multiplication by $\pm 1$ by the following properties:
\begin{itemize}
    \item $\div(\phi_m) = \sum_{Q\in E[m]} (Q) - m^2(\infty)$;
     \item $\phi_m^2$ is a polynomial in $x$ only with leading term $m^2x^{m^2-1}$ and coefficients in $\Z[A_2,A_4, A_6]$.
    \end{itemize}

The value $\lambda_p(P)$ for $P\not \in E_f(\Q_p)$ as above is independent of the branch of the $p$-adic logarithm and the definition can be extended to torsion points by continuity. 

We want to study $\lambda_p$ as a function on a residue disc. By \eqref{eq:lambadp_inf}, if $T$ is a local coordinate at the point at infinity, $\lambda_p$ can be expressed in terms of $T$ as $-2\log(f(T))$ where $f(T)\in T\Z_p[[T]]^{\times}$. Thus, the analysis can be obtained by applying Lemma \ref{lemma:auxiliary} and is similar to \S \ref{subsec:prec_log}. We postpone it until we also understand the argument of the logarithm in \eqref{eq:lambdapaff} as a power series. 

So let $P(T)$ be the parametrisation of a disc on an elliptic curve not reducing to the point at infinity modulo $p$ and let $m $ such that $mP(T)$ reduces to infinity for all $T \in p \Z_p$.

    \begin{lemma}
    \label{lemma:arg_lambda}
    With the above assumptions, we have
    $$\frac{\sigma_p(mP(T))}{\phi_m(P(T))} = c_0 + O(T)\in\Z_p[[T]]\qquad \text{with } c_0\in\Z_p^{\times}.$$
    \end{lemma}
    \begin{proof}
Since $\phi_m(P(T))\in \Z_p[[T]]$, by the $p$-adic Weierstrass preparation theorem we may factor it as \begin{equation*}
    p^{n}F(T)u(T)
\end{equation*} 
where $n$ is an integer, $u(T)\in \Z_p[[T]]^{\times}$ is a unit power series and $F(T)$ is a distinguished polynomial. In particular, the zeros in $\overline{\Q_p}$ of $F(T)$ all have positive valuation, and hence correspond to $m$-torsion points in the disc of $P(T)$ (since $P(T)$ converges on $|T|_p<1$).

Now we turn our attention to $\sigma_p(mP(T))$. First note that $mP(T)$ may be viewed as a point of $E$ over $\Frac(\Z_p[[T]])$ (since $x(T),y(T)\in \Z_p[[T]]$). Thus, we may use the $p$-adic Weierstrass preparation theorem to study
\begin{equation*}
    \tau = -\frac{x(mP(T))}{y(mP(T))}\in \Frac(\Z_p[[T]]).
\end{equation*}
We would like to show that $\tau \in \Zp[[T]]$ with constant term divisible by $p$. 
Since $\tau$ has no poles for $|T|_p<1$, we must have $\tau\in \Q\cdot \Z_p[[T]]$.
For every $|T_0|_p<1$, we have $|\tau(T_0)|_p<1$; hence picking $T_0$ with $|T_0|_p$ close enough to $1$ shows that $\tau \in \Z_p[[T]]$. Indeed, write $\tau = \alpha g(T)$, where $\alpha = p^{k}$ and $g(T) = \sum_{i = 0}^{\infty} c_i T^i\in \Z_p[[T]]\setminus p\Z_p[[T]]$. Let $r$ be minimal such that $\ord_p(c_r) = 0$. Then, for every $T_0$ with $|T_0|_p < 1$ and for every $i> r$
\begin{equation*}
|c_i T_0^i|_p\leq |T_0^i|_p < |T_0^r|_p = |c_r T_0^r|_p.
\end{equation*}
Furthermore, if $|T_0|_p>|c_i|_p^{1/(r-i)}$ for all $i< r$, then for all $i<r$ we have
\begin{equation*}
|c_iT_0^i|_p <|T_0|_p^r = |c_r T_0^r|_p;
\end{equation*}
hence, $|g(T_0)|_p = |T_0^r|_p$. If $k < 0$, we may pick $T_0$ that further satisfies $|T_0^r|_p>|\alpha^{-1}|_p$, which gives
\begin{equation*}
|\tau(T_0)|_p = |\alpha T_0^r|_p> 1,
\end{equation*}
a contradiction. 
Finally, since $\tau(0) \in p\Z_p$, we have the claim on the constant term.

So $\sigma_p(mP(T))\in \Z_p[[T]]$; it has simple zeros in the open unit disc precisely at the $m$-torsion points in the disc $P(T)$. Comparing with $\phi_m(P(T))$, we get that
\begin{equation*}
\frac{\sigma_p(mP(T))}{\phi_m(P(T))} \in \Q\cdot \Z_p[[T]].
\end{equation*}
Moreover this quotient is non-vanishing on $|T|_p<1$, hence in fact it is, up to multiplication by $p^k$ for some $k\in \Z$, a unit power series. This implies that at a given $T_0$ in the open unit disc, the valuation of the quotient is $-k$. But picking $T_0\in p\Z_p$ such that $mP(T_0) \neq \infty$, we have that $\ord_p(\phi_m(P(T_0))) = \ord_p(d(mP(T_0))) = \ord_p(T(mP(T_0))) = \ord_p(\sigma(mP(T_0)))$, hence $k = 0$. Here $d$ denotes the square-root of the denominator of the $x$-coordinate and the first equality holds true because $p$ is a prime of good reduction (see e.g.\ \cite[Proposition 1]{wuthrich}). 
    \end{proof}
    
    \begin{cor}\label{cor:lambda} 
    Let $T$ be a local coordinate at a point $P\in E(\Q_p)$, which we assume either at infinity or not in the disc of infinity and let $m$ be the order of the reduction of $P$ modulo $p$. Then
    \begin{equation*}
     \lambda_p(P(T)) = -2\delta\log(T) + \sum_{n=0}^{\infty} C_n T^n,
    \end{equation*}
    where $\delta = 1$ if $P$ is the point at infinity and $0$ otherwise, and $\ord_p(C_0)\geq 1 -2\ord_p(m)$, $\ord_p(C_n)\geq -\ord_p(n)-2\ord_p(m)$ for $n\geq 1$.
    \end{cor}
    
    \begin{proof}
    In view of Lemma \ref{lemma:arg_lambda} and the considerations preceding it, the corollary follows by Lemma \ref{lemma:auxiliary}, similarly to the proof of Lemma \ref{lemma:log}.
    \end{proof}
    
    \begin{rmk}
    While we found it convenient to work out precision estimates for $\lambda_p$ using multiplication-by-$m$ on $E/\Frac(\Z_p[[t]])$, we found it computationally more convenient to use the following formula for $\lambda_p(P(T))$. Let  $E_2(E,\omega)$ be the value of the weight two Katz Eisenstein series \cite{Katz, Katzinterpolation} at the pair $(E,\omega)$ and let
    \begin{equation*}
     c = \frac{4A_2 - E_2(E,\omega)}{12}\in \Q_p, \qquad \eta = (x+c)\frac{dx}{2y}.
    \end{equation*}
    Then, setting $P_0 = P(0)$, we have
    \begin{equation}\label{eq:lambda_p_break}
        \lambda_p(P(t)) = \lambda_p(P_0) + 2\int_{P_0}^{P(t)} \omega_0\eta + 2\int_{\infty}^{P_0}\eta \cdot \int_{P_0}^{P(t)} \omega.
    \end{equation}
    The integrals involving $P(t)$ are all formal integrals; the remaining integral of $\eta$ is a Coleman integral of a differential of the second kind and can be computed using division polynomials (or the Coleman integration algorithm of \cite{BBK10}); we omit details, but formula \eqref{eq:lambda_p_break} can be derived from \cite[(4.1)]{QC0} invoking suitable height comparison results. 
    
If $p\geq 5$, we can compute $E_2(E,\omega)$ and $\lambda_p(P_0)$ using an algorithm of Harvey \cite{harvey}, which builds on one by Mazur--Stein--Tate \cite{MST} and is implemented in \texttt{SageMath} \cite{sage}. If $p=3$, we use an algorithm of Balakrishnan \cite{Bal3adic}, available at \cite{Bal3adiccode}.
    \end{rmk}

\subsection{Precision of $\tilde{\rho}$}
We can now use the considerations of \S\S \ref{subsec:prec_log}, \ref{subsec:prec_Log}, \ref{subsec:prec_lambda} to deduce lower bounds for the $p$-adic valuation of the coefficients of the expansion of $\tilde{\rho}(z)$ in a residue disc.

\begin{lemma}\label{lemma:unram}
Let $z\in X(\Q_p)$ such that $\overline{z}$ is not a ramification point of $\varphi_i$ and let $t$ be a local coordinate at $z$. Let $T_i$ be a local coordinate for $\varphi_i(z)$. Then 
\begin{equation*}
T_{i}(\varphi_i(z(t))) = t\cdot u(t) \qquad \text{for some}\quad u(t)\in \Z_p[[t]]^{\times}.
\end{equation*}
\end{lemma}
\begin{proof}
Since $\varphi_i$ is unramified at $z$, we have that $\varphi_i^{*}T_{i}$ is a uniformiser for $z$. This applies to $X/\Q_p$, as well as $X/\F_p$, so $\varphi_i^{*}T_{i}$ is a local coordinate. 
\end{proof}

\begin{lemma}\label{lemma:ram}
Let $z\in X(\Q_p)$ such that $\overline{z}$ is a ramification point for $\varphi_i$ and let $t$ be a local coordinate at $z$. Let $T_i$ be a local coordinate for $\varphi_i(z)$. Then
\begin{equation*}
T_i(\varphi_i(z(t))) = t^2\cdot u(t) \qquad \text{for some}\quad u(t)\in \Z_p[[t]]^{\times}.
\end{equation*}
\end{lemma}

\begin{proof}
The ramification index of any such point is $2$.
\end{proof}

\begin{rmk}
The ramification points of $\varphi_1$ are those satisfying $x = 0$; the ramification points of $\varphi_2$ are the points at infinity.
\end{rmk}

\begin{proposition}\label{prop:conclusions}
Let $t$ be a local coordinate at a point $z\in X(\Q_p)$ as in \S \ref{subsec:prec_intro} and let
\begin{equation*}
\epsilon = \min\{ \ord_p(\alpha_1),-2\ord_p(\#E_1(\F_p)), \ord_p(\alpha_2),-2\ord_p(\#E_2(\F_p))\}. 
\end{equation*}
Then, under
Assumption \ref{ass:ordinary} on ordinarity,
\begin{equation*}
\tilde{\rho}(z(t)) = \sum_{n=0}^{\infty} C_n t^n\in \Q_p[[t]]
\end{equation*}
with $\ord_p(C_0), \ord_p(C_1)\geq \epsilon$ and, for all $n\geq 2$, $\ord_p(C_n)\geq -\lfloor{\log_p(n-1)\rfloor}-\ord_p(n) + \epsilon$. Moreover, if $p\nmid \# E_1(\F_p)\cdot \# E_2(\F_p)$, then $\ord_p(C_0)\geq 1+\epsilon$.   
\end{proposition}
\begin{proof}
This follows from Lemmas \ref{lemma:log}, \ref{lemma:Log}, Corollary \ref{cor:lambda} and Lemmas \ref{lemma:unram}, \ref{lemma:ram}, in view of Lemma \ref{lemma:auxiliary}.
\end{proof}
 
Let $M\geq 2$ be an integer and let $\tilde{\rho}_M(t)\in \Q_p[t]$ be a polynomial of degree less than $M$ such that
\begin{equation*}
\tilde{\rho}_M(t) - \tilde{\rho}(z(t)) = O(t^{M}).
\end{equation*}

\begin{lemma} Let $N =M- \lfloor{\log_p(M-1)\rfloor}-\lfloor{\log_p(M)\rfloor}+ \epsilon$. Then
\begin{equation*}
\tilde{\rho}_M(pt) - \tilde{\rho}(z(pt)) = O(p^{N}).
\end{equation*}
\end{lemma}
\begin{proof}
Let $n\geq M$. By Proposition \ref{prop:conclusions}, the coefficient of the term of degree $n$ in $\tilde{\rho}(z(pt))$ has valuation bounded from below by
\begin{equation*}
n -\lfloor{\log_p(n-1)\rfloor}-\ord_p(n) + \epsilon\geq n -\lfloor{\log_p(n-1)\rfloor}-\lfloor{\log_p(n)\rfloor} + \epsilon.
\end{equation*}
The right hand side of this inequality is (not necessarily strictly) increasing for $n$ an integer $\geq 2$.
\end{proof}

Let $w\in \Omega$. Let $k$ be the minimal valuation of a coefficient of $\tilde{\rho}_M(pt)-w$. If $k< N$, then $\tilde{\rho}_M(pt)-w$ is non-zero modulo $N$. If $p^{-k}(\tilde{\rho}_M(pt)-w)$ has a zero in $\Z/p^{N-k}\Z$ whose derivative is non-zero modulo $p^{\lceil (N-k)/2\rceil}$, then by Hensel's lemma it lifts uniquely to a zero of $\tilde{\rho}(z(pt))-w$ in $\Z_p$ (we use \cite[Theorem 8.2 (1)]{conrad:hensel} to determine to which precision we know the lift).

Furthermore, any zero in $\Z_p$ of $\tilde{\rho}(z(pt))$ reduces modulo $p^{N-k}$ to a zero of $p^{-k}\tilde{\rho}_M(pt)$. When it comes to zeros not corresponding to known rational points, the uniqueness of the lifting is not so important. Indeed, suppose that we found a root modulo $p^{N^{\prime}}$ that could or could not lift (perhaps not uniquely), and that we suspect does not correspond to a rational point of $X$. Then in the Mordell--Weil sieve step (\S\S \ref{subsec:step2}, \ref{subsec:step3}) we will try to show that such a root cannot possibly correspond to a point in $X(\Q)$.

On the other hand, it is important to show that the zeros corresponding to our known rational points are isolated, so that we can discard such roots at once before the Mordell--Weil sieve step. Because of the extra automorphisms $X$ possesses, we actually expect some of these to be non-simple. 

\begin{prop}\label{prop:double_roots}
Let $z\in X(\Q)$ such that $x(z) = 0$, or $y(z) = 0$, or $z$ is a point at infinity and let $t$ be the local coordinate at $z$ of \S \ref{subsec:prec_intro}. Let $w\in \Omega$ such that $\tilde{\rho}(z) = w$. Then
\begin{equation*}
\tilde{\rho}(z(t)) - w = t^2f(t) \qquad \text{for some } f(t)\in \Q_p[[t]]. 
\end{equation*}
\end{prop}

\begin{proof}
The point $z$ is fixed by the hyperelliptic involution or one of the automorphisms $(x,y)\mapsto (-x,y)$ and $(x,y)\mapsto (-x,-y)$. Let $\theta$ be the automorphism fixing $z$. Then $\rho(\theta(z)) = \rho(z)$, since upon noticing that $\varphi_i\circ\theta$ is either the identity or multiplication by $-1$, we see that each of the terms making up $\tilde{\rho}(z)$ is invariant under $\theta$. Now, in view of our choice of $t$, we have $\theta^{*}t = -t$. Therefore,
\begin{equation*}
\tilde{\rho}(z(t)) =  \tilde{\rho}(z(-t)),
\end{equation*}
and hence $\tilde{\rho}(z(t))$ has trivial coefficient of $t^{2k+1}$ for all non-negative $k$. Finally since $\tilde{\rho}(0) = w$, the proposition follows.
\end{proof}

Our strategy is then the following:
\begin{enumerate}[label=(\roman*)]
    \item\label{it:double_root_auto} If $\tilde{\rho}_M(pt)-w = O(t^2)$, then we verify that we are in the situation of Proposition \ref{prop:double_roots}. We do this by checking that the coefficient of $t^2$ is non-zero, and that $t=0$ corresponds to a point at infinity, or with vanishing $x$- or $y$-coordinate. 
    \item\label{it:lift_with_Hensel} In the other cases, we check that the recovered roots can be lifted uniquely using Hensel's lemma.
    \item We return an error if a root does not meet the conditions of \ref{it:double_root_auto} or \ref{it:lift_with_Hensel}.
\end{enumerate}

\section{Computations}\label{sec:computations}
We now explain how we used the results of Sections \ref{sec:QC} and \ref{sec:prec_an} together with the Mordell--Weil sieve to prove Theorem \ref{thm:db}, that is, to compute the rational points on a database $\Delta$ of genus $2$ curves over $\Q$ that have points everywhere locally, are bielliptic and have a rank $2$ Jacobian.  

We describe three main steps in our implementation. First, we explain in \S \ref{subsec:step1} how we obtained the dataset $\Delta$, and how we computed, for each curve $X\in \Delta$, various inputs for the quadratic Chabauty and Mordell--Weil sieve computations. For instance, we need to choose a set of primes with respect to which to apply the quadratic Chabauty technique described in the previous sections. For a given prime, we discussed in Section \ref{sec:prec_an} the theoretical results needed for an implementation of the quadratic Chabauty method. Therefore, we only explain here (\S \ref{subsec:step2}) how to turn the output of this computation into an input for a Mordell--Weil sieve. Finally in \S \ref{subsec:step3}, we explain how to use the Mordell--Weil sieve, and, in particular, its \texttt{Magma} \cite{magma} implementation available at \cite{QCMagma}, to complete the determination of the rational points on $X$.

The remaining \S \ref{subsec:pointless} and \S \ref{subsec:sharp} concern some exceptional curves in our database. 
\subsection{Step 1: The database and some preliminary computations}\label{subsec:step1} All the computations of this step were performed in \texttt{Magma}\cite{magma} (see \cite{RationalPts}).
Recall that our database contains $413$ curves in total, $411$ of which belong to the database \cite{genus2curvedatabase}, available on the LMFDB \cite{lmfdb}. The latter 411 curves were extracted from the 66,158 genus $2$ curves of the LMFDB as follows.  

First, we are interested in curves that are locally solvable, and we can filter the LMFDB search to return such curves only. Secondly, a genus $2$ curve $X$ over $\Q$ is bielliptic (over $\Q$) if and only if the $\Q$-automorphism group of $X$ has a subgroup isomorphic to $\Z/2\Z\times \Z/2\Z$. Finally, if $X$ is bielliptic, there exist elliptic curves $E_1$ and $E_2$ such that the Jacobian $J$ of $X$ is (Richelot) isogenous to $E_1\times E_2$ \cite{Richelot_36,Richelot_37,cassels_flynn,Smith_thesis}; therefore, the Mordell--Weil rank of $J$ is equal to the sum of the ranks $r_1$ and $r_2$ of $E_1$ and $E_2$. We use this to identify the curves with a rank $2$ Jacobian. 

If $r_1$ or $r_2$ is equal to zero, we may apply elementary methods to determine $X(\Q)$. Therefore, we hereafter restrict our attention to the $354$ curves $X\in \Delta$ for which $r_1 = r_2 = 1$.
For each such $X$, we compute the following data.
\subsubsection*{A ``bielliptic" model} The equation for $X$ as given in the LMFDB is of the form $y^2 + f_1(x)y = f_2(x)$, for some $f_1(x),f_2(x)\in\Z[x]$ of degrees at most $3$ and $6$, respectively. In order to apply Theorem \ref{thm:main}, we need to find a \textit{bielliptic} model of the form $y^2 = a_6x^6 + a_4x^4 + a_2 x^2 + a_0\in \Z[x]$, or, equivalently, models for $E_1$ and $E_2$ of the form
\begin{align*}
E_1\colon y^2 = x^3 + a_4x^2 + a_2a_6x + a_0a_6^2,\\
E_2 \colon y^2 = x^3 + a_2x^2 + a_4a_0x + a_6 a_0^2.
\end{align*}
Equations for $E_1$ and $E_2$ of this form (though not necessarily integral) are computed internally by \texttt{Magma}'s function \texttt{RichelotIsogenousSurfaces}. We use this to compute an integral bielliptic model for $X$.
\subsubsection*{A candidate list of rational points} Let $X(\Q)_{\known}$ be the set of rational points of $X$ such that the naive height of the $x$-coordinate, with respect to the bielliptic model from above, is less than $10^4$. We compute $X(\Q)_{\known}$ using the \texttt{Magma} function \texttt{RationalPoints}. The ultimate goal of our computation will be to prove that $X(\Q) = X(\Q)_{\known}$.

\subsubsection*{Generators of $J(\Q)$} In order to apply Theorem \ref{thm:main}, we only need to know a point of infinite order on each of $E_1$ and $E_2$. However, for the subsequent Mordell--Weil sieve step, we assume that we know generators for the full Mordell--Weil group $J(\Q)$. The \texttt{Magma} function \texttt{MordellWeilGroupGenus2}, implemented by Stoll, successfully determined these for every curve. We denote by $B_1$ and $B_2$ generators for $J(\Q)/J(\Q)_{\tors}$. 

\subsubsection*{A set of primes for quadratic Chabauty} In order to apply Theorem \ref{thm:main}, we need to pick a prime $p$ of good reduction. In addition, in the precision estimates of \S \ref{subsec:prec_lambda}, we assumed  that each of $E_1$ and $E_2$ has ordinary reduction at $p$. In general, it is expected that Theorem \ref{thm:main} will not suffice by itself to determine $X(\Q)$, since it will only return a $p$-adic approximation of a superset of $X(\Q)$.  The strategy that we will describe in detail in the next steps entails using the Mordell--Weil group $J(\Q)$ and reduction maps to $J(\F_{\ell})$, for various primes $\ell$, to refine the superset and prove that $X(\Q) = X(\Q)_{\known}$.  This method is often more likely to succeed if the quadratic Chabauty computation is performed for more than one prime $p$. 
We therefore compute, for each curve, the three smallest primes $p_1,p_2,p_3$ of good ordinary reduction. See Remark \ref{rmk:choosing_qc_primes} for possible improvements in the choice of primes.

\subsection{Step 2: Extra points in bielliptic quadratic Chabauty}\label{subsec:step2}
In the following discussion we assume that our curve has at least one known $\Q$-rational point $b$. While some modification of this would be applicable in the other case too, we decided to treat curves with no known rational points separately: see \S \ref{subsec:pointless}. Our code for this stage of the computation is written in \texttt{SageMath} \cite{sage} and is available at \cite{OurCode}.

We perform the bielliptic quadratic Chabauty algorithm on $X$ for the three primes $p_1,p_2,p_3$. 
Fix $i\in\{1,2,3\}$. Using the precision estimates of Section \ref{sec:prec_an}, we compute the set
\begin{equation*}
A_i = \{z\in X(\Q_{p_i}):\tilde{\rho}_i(z) \in \Omega_i\},
\end{equation*}
where $\tilde{\rho}_i(z)$ and $\Omega_i$ are the function and set from Theorem \ref{thm:main}, respectively, for the prime $p_i$.
By Theorem \ref{thm:main}, the set $A_i$ contains $X(\Q)$; it may or may not contain other points in $X(\Q_{p_i})$. Let $A_{\text{extra},i }$ be the set of $p_i$-adic points in $A_i$ which have not been recognised as points in $X(\Q)_{\known}$. The points in $A_{\text{extra},i}$ are only known modulo $p_i^{m_i}$, for some integer $m_i$ depending on our chosen working precision. 
If $A_{\text{extra},i}\neq \emptyset$, our strategy to \emph{prove} that the points in $A_{\text{extra},i}$ are not reductions modulo $p_i^{m_i}$ of points in $X(\Q)$ is to feed them into the Mordell--Weil sieve (described in \S \ref{subsec:step3}), following the strategy outlined in \cite[Sections 5-7]{BBM17}. We describe here the preliminary step (which essentially amounts to applying Section 6 of \emph{loc.\ cit.} to our setting).

Let $\iota$ be the Abel--Jacobi map on $X(\Q)$ with respect to $b$:
\begin{equation*}
\iota \colon X(\Q) \xhookrightarrow{} J(\Q), \qquad P \mapsto [P-b].
\end{equation*}
Let $P_i\in A_{\text{extra},i}$ be one of our extra points and assume for the sake of contradiction that $P_i$ corresponds to a point $P\in X(\Q)$. Then there exist integers $a_1,a_2 \in \Z$ and a torsion point $T\in J(\Q)$ such that
\begin{equation}\label{eq:coeff_iota_P}
\iota(P) = a_1 B_1 + a_2 B_2 + T.
\end{equation}
We can compute $a_1,a_2$ modulo $p_i^{n_i}$, for some $n_i\leq m_i$, by noting that, for each $j\in\{1,2\}$, we have 
\begin{equation*}
\Log(\varphi_j(P)) - \Log(\varphi_j(b)) = a_1 \Log(\varphi_{j,*}(B_1)) + a_2\Log(\varphi_{j,*}(B_2)).
\end{equation*}

Therefore, for every $P_i\in A_{\text{extra},i}$, we obtain at most $\#J(\Q)_{\tors}$ possibilities for the image of $\iota(P)$ in $J(\Q)/p_i^{n_i}J(\Q)$, under our running assumption that $P_i$ corresponds to a point $P\in X(\Q)$.

The goal of the sieve is to show that such cosets in $J(\Q)/p_i^{n_i}  J(\Q)$ cannot arise from points in $X(\Q)$ by considering reduction maps modulo several primes.  

We can consider more than one quadratic Chabauty prime in $\{p_1,p_2,p_3\}$ as follows. Suppose $k\neq i$. Then there must also exist $P_k\in A_{\text{extra}, k}$ such that $P$ reduces modulo $p_k^{m_k}$ to $P_k$. Using the Chinese remainder theorem, we then obtain cosets of $p_i^{n_i}p_k^{n_k}J(\Q)$ in $J(\Q)$ that we want to eliminate in the Mordell--Weil sieve step.

We refine this strategy by noticing that some elements in $A_{\text{extra}, k}$ may be ruled out from corresponding to $P$ in the following way.
First note that elements of $\Omega_i$ are sums over the bad primes $q$ of $\Q$-rational multiples of $\log_i(q)$, where $\log_i$ is the $p_i$-adic logarithm, and that for $k\neq i$ the set $\Omega_k\subset \Q_{p_k}$ can be obtained from $\Omega_i\subset \Q_{p_i}$ just by replacing each occurrence of $\log_i$ with $\log_k$. Moreover, if $P_i\in A_{i}$ and $P_k\in A_{k}$ both correspond to $P\in X(\Q)$, we must have
\begin{equation*}
\tilde{\rho}_i(P) = \sum_{q\: \text{bad}}v_q\log_i(q),\qquad  \tilde{\rho}_k(P) = \sum_{q\: \text{bad}}v_q\log_k(q),
\end{equation*}
where, for every $q$, the number $v_q$ is rational (the same one for $i$ and $k$). It follows that we only need to compare points in $A_{\text{extra},i}$ and $A_{\text{extra},k}$ corresponding to compatible elements in $\Omega_i$ and $\Omega_k$.

Since $A_i$ is closed under the hyperelliptic involution and the automorphisms $(x,y)\mapsto (-x,\pm y)$ of $X$, for one of the three primes it suffices to compute $A_i$ modulo automorphisms. We do so for $p_1$.

\subsection{Step 3: The Mordell--Weil sieve}\label{subsec:step3}
The Mordell--Weil sieve is a powerful tool for obtaining information about the rational points of a curve of genus greater than $1$. It first appears in Scharaschkin's Phd thesis \cite{Scharaschkin99} as a strategy to prove that a curve has no rational points.  
However, it can also be used in conjunction with methods such as classical or quadratic Chabauty to determine the set of rational points when we know that this is non-empty. See for instance \cite{Flynn_Hasse, Poonen_Schaefer_Stoll, Bruin_Stoll_existence, BS10} for successful sieving computations. Here we apply to the curve $X$ of \S \ref{subsec:step2} the technique of \cite[Sections 5-7]{BBM17} and the \texttt{Magma} implementation thereof available at \cite{QCMagma}; see \cite{RationalPts}.

We retain the notation of \S \ref{subsec:step2}. Let $I$ be a subset of $\{1,2,3\}$, and let $M=\prod_{j\in I} p_j^{n_j}$. In \S \ref{subsec:step2}, we reduced the problem of determining $X(\Q)$ to that of showing that some subset $C_M$ of the quotient $J(\Q)/MJ(\Q)$ does not contain the image of a point in $X(\Q)$, under the composition of the embedding $\iota$ with the canonical quotient map $\pi\colon J(\Q)\to J(\Q)/MJ(\Q)$. More generally, by applying the Chinese remainder theorem, we can work with $M =M^{\prime}\prod_{j\in I} p_j^{n_j} $ for some integer $M^{\prime}$ coprime to $p_j$ for all $j\in I$.

Let $S$ be a finite set of primes of good reduction for $X$ and consider the commutative diagram
\[\begin{tikzcd}
    X(\Q) \ar[r,"\pi \circ \iota"] \ar[d] & J(\Q)/MJ(\Q) \ar[d,"\alpha_S"] \\
    \displaystyle\prod_{\ell \in S}X(\F_{\ell}) \ar[r,"\beta_S"] & \displaystyle\prod_{\ell \in S}J(\F_{\ell})/MJ(\F_{\ell}),
\end{tikzcd}\]
where the vertical and bottom maps are the natural ones. 
Because of the commutativity of the diagram, we succeed in determining $X(\Q)$ if we can find a set $S$ for which  
\begin{equation}
\label{emptyness}
 \alpha_S(C_M) \cap \beta_S \left(\prod_{\ell \in S} X(\F_{\ell})\right) = \emptyset.
\end{equation}
We use the repository \cite{QCMagma} to choose a set $S$ of primes in such a way that \eqref{emptyness} has some likelihood of holding. 
The strategy is to pick primes $\ell$ for which the order of  $J(\F_{\ell})/MJ(\F_{\ell})$ is large relative to the number of points in $X(\F_{\ell})$ (or relative to its Hasse--Weil approximation $\ell +1 $). In particular, we pick primes $\ell\leq 10^{4}$ for which
\begin{equation*}
   \frac{ \# J(\F_{\ell})/MJ(\F_{\ell})}{ \ell + 1}> 2.
\end{equation*}
\begin{rmk}\label{rmk:choosing_qc_primes}
We could (but did not) apply a similar strategy in the choice of the primes $p_1,p_2,p_3$. Namely, we could restrict to primes $p$ for which, for some integer $n$, the ratio $\frac{ \# J(\F_{\ell})/p^n J(\F_{\ell})}{ \ell + 1}$ is large, for at least one choice of $\ell \leq 10^4$. We refer the reader to \cite[Section 7]{BBM17} for a discussion on how to simultaneously make optimal choices for the set $S$ and the integer $M$.
\end{rmk}

Our cosets $C_M$ are naturally partitioned into $\#\Omega_1 = \#\Omega_2= \#\Omega_3$ subsets (as explained in \S \ref{subsec:step2}). We run a separate sieve for each such subset. 

For $M^{\prime}\in [1,2,4]$, we do the following:
\begin{enumerate}
\item Let $S$ be a suitable set of primes (in the sense above) for $M^{\prime}p_1^4p_2^4p_3^4$.
\item For $I\in [\{1,2\},\{1,3\}, \{2,3\}, \{1,2,3\}]$, run the sieve with respect to $M=M^{\prime} \prod_{j\in I} p_j^4$ and the set $S$. If \eqref{emptyness} holds for some $I$, stop.
\end{enumerate}

Of course, even if the sieve does not succeed in eliminating the whole of $C_M$, we should not discard the information on which cosets are sieved out when changing $M$. 

If $J(\Q)_{\tors}$ is non-trivial, we may also try to take $M^{\prime}$ to be a divisor of $\#J(\Q)_{\tors}$.

For our database, the choices of $n_j = 4$ for each $j\in\{1,2,3\}$ and $M^{\prime}\in \{1,2,4\}\cup \{\text{divisors of }\#J(\Q)_{\tors}\}$ were successful for every curve. If one were to apply the technique to one specific curve (rather than to a database), it might be advisable to make more ad hoc choices. 

\subsection{Curves with no known rational points}\label{subsec:pointless}
In \S \ref{subsec:step2}--\ref{subsec:step3}, we explained our strategy for determining $X(\Q)$ when $X(\Q)_{\known} \neq \emptyset$. If $X(\Q)_{\known} = \emptyset$, we skip the quadratic Chabauty computations and directly apply a Mordell--Weil sieve to prove that $X(\Q) = \emptyset$. To this end, we implemented \cite{RationalPts} a simple sieve that uses only one good prime $\ell$ and the fact that $X$ is bielliptic (cf.\ \cite[Example 8.3]{siksek_sieve}). Let  $\varphi = (\varphi_1,\varphi_2)$. Then we have a commutative diagram
\begin{equation*}
\begin{tikzcd}
    X(\Q) \ar[r,"\varphi"] \ar[d, "\red_{\ell}"] & E_1(\Q)\times E_2(\Q)\ar[d, "\red_{\ell}"] \\
   X(\F_{\ell}) \ar[r,"\varphi"] & E_1(\F_{\ell})\times E_2(\F_{\ell}),
\end{tikzcd}
\end{equation*}
where $\red_{\ell}$ denotes reduction modulo $\ell$. 
If $\varphi(X(\F_{\ell}))\cap \red_{\ell}(E_1(\Q)\times E_2(\Q))$ is the empty set, we are done.

Our implementation relies on functions in \cite{BS10_code} for constructing explicit maps  $E_i(\F_p)\xrightarrow{\sim} A_i$, where $A_i$ is an abstract abelian group. 

We applied this to two curves in $\Delta$: the ones with LMFDB labels \href{http://www.lmfdb.org/Genus2Curve/Q/473256/a/946512/1}{473256.a.946512.1} and \href{http://www.lmfdb.org/Genus2Curve/Q/826672/a/826672/1}{826672.a.826672.1}. We used the prime $\ell = 331$ and $\ell = 181$, respectively. These two curves are examples of violation of the Hasse principle.

\begin{rmk}
We could have replaced $J$ with $E_1\times E_2$ also in the commutative diagram of \S \ref{subsec:step3}. We chose not to do so, since the code \cite{QCMagma} was directly applicable. 
However, the fact that $X$ is bielliptic is implicitly used. Indeed, an important ingredient in \S \ref{subsec:step3} is the computation of the Mordell--Weil group $J(\Q)$. In the bielliptic case, the \texttt{Magma} implementation for this uses the isogeny $J\sim E_1\times E_2$.
\end{rmk}

\subsection{Sharp quadratic Chabauty computations}\label{subsec:sharp} For $6$ curves we succeeded in determining the set of rational points using quadratic Chabauty only (i.e.\ without performing \S \ref{subsec:step3}). This may be of some interest in the context of conjectures on sharpness of more general Chabauty--Kim sets (cf. \cite[Conjecture 3.1]{BDCKW}). 

The curves, together with the relevant primes used for quadratic Chabauty, are listed in the following table.
\begin{center}
\begin{tabular}{ |c|c|c|c|} 
 \hline
 LMFDB label & Bielliptic model & Prime(s)  & $\#X(\Q)$\\ 
 \hline
  \href{http://www.lmfdb.org/Genus2Curve/Q/99856/b/99856/1}{99856.b.99856.1} & $y^2 = x^6 + 22x^4 - 19x^2 + 4$ & $p_1 = 3$ & 8 \\
 \hline
 \href{https://www.lmfdb.org/Genus2Curve/Q/322624/b/322624/1}{322624.b.322624.1} & $y^2 = x^6 - 2x^4 - 7x^2 + 4$ &  $p_1 = 3$ & 8\\ 
 \hline
 \href{https://www.lmfdb.org/Genus2Curve/Q/614656/a/614656/1}{614656.a.614656.1} & $y^2 = x^6 - 83x^4 + 19x^2 - 1$ & $p_1 = 3$ & 6 \\ 
 \hline
 \href{https://www.lmfdb.org/Genus2Curve/Q/571536/a/571536/1}{571536.a.571536.1} & $y^2 = x^6 - 12x^4 + 36x^2 - 4$ & $p_1 = 5$ & 2  \\
 \hline
\href{https://www.lmfdb.org/Genus2Curve/Q/274576/a/274576/1}{274576.a.274576.1} & $y^2 = x^6 - 4x^4 - 4x^2 - 4$ & $p_1 = 3,p_2 = 7$ & 2 \\
  \hline
 \href{https://www.lmfdb.org/Genus2Curve/Q/489648/a/489648/1}{489648.a.489648.1} & $y^2 = -3x^6 + 4x^4 + 4x^2 - 4$ & $p_1 = 5$ & 4  \\
 \hline
\end{tabular}
\end{center}
In particular, for the first four listed curves, the set $A_{\text{extra},1}$ is empty (cf.\ \S \ref{subsec:step2}). For \href{https://www.lmfdb.org/Genus2Curve/Q/274576/a/274576/1}{274576.a.274576.1}, neither $A_{\text{extra},1}$, nor $A_{\text{extra},2}$ is empty; however, there exists no pair $(\omega_1,\omega_2)\in\Omega_1\times \Omega_2$ of compatible elements for which both $A_{\text{extra},1}$ and $A_{\text{extra},2}$ contain a point. 

Finally, for \href{https://www.lmfdb.org/Genus2Curve/Q/489648/a/489648/1}{489648.a.489648.1}, for each point $P_1$ in the (non-empty) $A_{\text{extra},1}$, we find that at least one of the coefficients $a_1,a_2$ in \eqref{eq:coeff_iota_P} has negative $p_1$-adic valuation, a contradiction to their being integers.

\bibliographystyle{amsalpha}
\bibliography{biblio}
\end{document}